\documentclass[10pt,a4paper]{amsart}
\usepackage[top=3.5cm,bottom=3.5cm,left=4cm,right=4cm]{geometry}
\usepackage{amsmath,amscd,amssymb,amsthm}
\usepackage[english]{babel}
\usepackage{lipsum}
\usepackage{mathtools}
\usepackage{enumerate}
\usepackage{setspace}
\usepackage{mathrsfs}
\usepackage[hidelinks]{hyperref}
\usepackage[all]{xy}
\usepackage{amsopn,amsfonts,pb-diagram}
\usepackage{stmaryrd,calc,t4phonet,enumerate,mathrsfs,empheq,wasysym,bbm}
\usepackage{listings}
\usepackage{hyperref}
\usepackage[numbers]{natbib}
\usepackage{bookmark}
\usepackage[OT2,T1]{fontenc}

\DeclareMathOperator{\tr}{Tr}

\DeclareMathAlphabet{\mathpzc}{OT1}{pzc}{m}{it}

\theoremstyle{theorem}
\newtheorem{theorem}{Theorem}[section]

\newtheorem{lemma}[theorem]{Lemma}
\newtheorem{proposition}[theorem]{Proposition}
\newtheorem{conjecture}[theorem]{Conjecture}

\newtheorem*{nota}{Notation}
\theoremstyle{definition}
\newtheorem{definition}[theorem]{Definition}
\newtheorem{remark}[theorem]{Remark}

\newtheorem{example}[theorem]{Example}

\newcommand{\bequ}{\begin{equation}}
\newcommand{\nequ}{\end{equation}}
\newcommand{\benu}{\begin{enumerate}}
\newcommand{\nenu}{\end{enumerate}}
\newcommand{\bnota}{\begin{nota}}
\newcommand{\nnota}{\end{nota}}
\newcommand{\bdia}{\[\begin{diagram}}
\newcommand{\ndia}{\end{diagram}\]}
\newcommand{\bcas}{\begin{cases}}
\newcommand{\ncas}{\end{cases}}
\newcommand{\barr}{\left(\begin{array}}
\newcommand{\narr}{\end{array}\right)}

\newcommand{\disp}{\displaystyle}
\newcommand{\sub}{\subseteq}

\newcommand{\N}{\mathbb{N}}
\newcommand{\Z}{\mathbb{Z}}

\newcommand{\F}{\mathbb{F}}

\author[A. Ferraguti]{Andrea Ferraguti}
\address{Institut f\"{u}r Mathematik\\
Universit\"{a}t Z\"{u}rich\\
Winterthurerstrasse 190\\
8057 Z\"{u}rich, Switzerland\\
}
\email{andrea.ferraguti@math.uzh.ch}
\author[G. Micheli]{Giacomo Micheli}
\address{Mathematical Institute\\
University of Oxford\\
Andrew Wiles Building\\
Radcliffe Observatory Quarter\\
Woodstock Road\\
OX2 6GG Oxford, UK
}
\email{giacomo.micheli@maths.ox.ac.uk}

\date{}
\title{On the existence of infinite, non-trivial $F$-sets}
\begin{document}

\begin{abstract}
In this paper we prove a conjecture of J.\ Andrade, S.\ J.\ Miller, K.\ Pratt and M.\ Trinh, showing the existence of a non trivial infinite $F$-set over $\mathbb F_q[x]$ for every fixed $q$. We also provide the proof of a refinement of the conjecture, involving the notion of width of an $F$-set, which is a natural number encoding the complexity of the set.
\end{abstract}
\maketitle

\section{Introduction}

Throughout this paper, $q$ is a prime power and $I_q$ is the set of all monic, irreducible polynomials in $\F_q[x]$.
\begin{definition}
An \emph{$F$-set} is a subset $A$ of $I_q$ such that for any $f(x)\in A$, all monic irreducible polynomials dividing $f(x)-f(0)$ are also in $A$.
\end{definition}
It is easy to construct finite $F$-sets but, on the other hand, it is not a priori clear whether there exist infinite $F$-sets which do not coincide with $I_q$. We will call an $F$-set \emph{non-trivial} if it is different from $I_q$. In this paper we are going to address \cite[Conjecture~1.2]{and}. Let us recall it here for completeness.
\begin{conjecture}\label{congettura}
For every prime power $q$, there exist an infinite, non-trivial $F$-set.
\end{conjecture}
In \cite{and} the authors provide nice constructions which solve the conjecture in the special cases of $q$ prime and congruent to $2$ or $5$ modulo $9$. In what follows we will prove both the conjecture and a stronger statement, which takes into account the cardinality of the prime divisors of elements of the form $f(x)-f(0)$, for $f$ in the $F$-set.

The paper is structured as follows.
In Section~\ref{conjecture}, we outline a proof of this conjecture, by explicitly exhibiting infinite non-trivial $F$-sets, sieving out the cases in terms of the factorization of $q-1$. The examples we produce are in some sense the easiest possible. This is made precise in Section~\ref{width}, where we introduce the notion of \emph{width} of an $F$-set. The width of an $F$-set is an element of $\N\cup\{\infty\}$ which measures the ``complexity'' of the $F$-set itself. For example, an $F$-set has width $0$ if and only if it is finite, whereas the $F$-sets constructed in Section~\ref{conjecture} have width $1$. Some properties of the width are proved in Proposition~\ref{basic_prop}. Section~\ref{technology} contains two technical lemmata that enable us to build $F$-sets of width $2$ and $\infty$.
Explicit examples of such sets are constructed in 
Section~\ref{main_section}.
We end the paper with a new Conjecture \ref{conjwithwidth} involving the notion of width of an $F$-set.

\section{Constructing infinite \texorpdfstring{$F$}{}-sets}\label{conjecture}
In this section, we  explain how to construct simple examples of infinite, non-trivial $F$-sets in $\F_q[x]$ for every prime power $q$. Recall that if $f(x)\in\F_q[x]$ is such that $f(0)\neq 0$, the \emph{order} of $f$ is defined as the smallest integer $e$ such that $f(x)\mid x^e-1$. See \cite[Lemma~3.1]{lid} for a proof of the existence of the order. In particular, let us recall 
\citep[Theorem~3.3]{lid} for completeness.
\begin{theorem}\label{root}
Let $f\in \F_q[x]$ be an irreducible polynomial of degree $m$ such that $f(0)\neq 0$ and let $\alpha\in \F_{q^m}$ be one of its roots. Then the order of $f$ equals the order of $\alpha$ in the multiplicative group $\F_{q^m}^*$.
\end{theorem}
The following is another classical result (see \cite[Theorem 3.35]{lid}) which will be useful later on.
\begin{theorem}\label{order}
Let $q$ be a prime power. Let $f(x)\in \F_q[x]$ be an irreducible polynomial of degree $m$ and order $e$. Let $t$ be a positive integer such that the prime factors of $t$ divide $e$ but not $(q^m-1)/e$. Assume also that $q^m \equiv 1 \bmod 4$ if $t \equiv 0 \bmod 4$. Then $f(x^t)$ is irreducible.
\end{theorem}
Finally, we recall another very nice result \cite[Proposition~2.3]{bos} by  Nigel Boston and Rafe Jones which characterizes stable degree $2$ polynomials. We state it in a slightly more specific form, which can be adapted from \citep[Theorem 2.2]{jon}.
\begin{theorem}\label{thm_composition}
Let $\F_q$ be a finite field of characteristic $\neq 2$, let $\gamma,m\in \F_q$ and let $f(x)\coloneqq(x-\gamma)^2+\gamma+m\in \F_q[x]$. For every $k\in\N$, let $f_k(x)$ denote the $k$-th fold composition of $f$ with itself. Then $f_k(x)$ is irreducible if and only if the set $\{-f(\gamma),f_2(\gamma),f_3(\gamma),\ldots,f_{k-1}(\gamma)\}$ does not contain any square.
\end{theorem}
\begin{proof}
In the statement of \citep[Theorem 2.2]{jon}, just observe that an element of a finite field is a square if and only if its norm is a square.
\end{proof}
\begin{theorem}\label{proof_of_conj}
Let $q$ be a prime power. Then there exists an infinite, non-trivial $F$-set in $\F_q[x]$.
\end{theorem}
\begin{proof}
When $q=2$, a non-trivial, infinite $F$-set is constructed in \cite[Theorem~1.1]{and}. 
Let now $q=3$ (or, more generally, suppose that $2$ is not a square in $\F_q$). Let $f(x)=x^2-2\in \F_3[x]$, and define the following sequence: $f_0(x)=x$ and $f_k(x)\coloneqq f(f_{k-1}(x))$ for every $k\in \N$. We claim that the set $A\coloneqq\{x,x+2,x-2\}\cup\{f_k(x)\}_{k\geq 1}$ is an infinite $F$-set. First, we have to check that $f_k(x)$ is irreducible for every $k$. This follows directly from Theorem~\ref{thm_composition} as $-f(0)=f_k(0)=2$ for every $k\geq 2$. Next, the reader should notice that $f_k(0)$ can be easily controlled for any $k$: $f_0(0)=0$, $f_1(0)=-2$ and finally $f_k(0)=2$ for any $k\geq 2$, as already observed.

We claim now that for $k\geq 2$ the factorization of $f_k(x)-2$ can be controlled as follows:
$$f_k(x)-2=(x-2)(x+2)f_0(x)^2\cdots f^2_{k-2}(x) \mbox{ for } k\geq 2.$$ 
Let us show this by induction. For $k=2$ we have $f_2(x)-2=(x^2-2)^2-4=(x-2)(x+2)f_0(x)^2$. Let the claim be true for $k$. We have that
\begin{align*}
f_{k+1}(x)-2=f^2_{k}(x)-4=(f_k-2)(f_k+2)=\\(x-2)(x+2)f_0(x)^2\cdots f^2_{k-2}(x)(f_{k-1}(x)^2-2+2),
\end{align*}
which completes the proof. Hence, $A$ is an infinite $F$-set and it is non-trivial as only three elements of $A$ have odd degree.

Finally, let $q$ be a prime power different from $2$ and $3$. Let $\alpha$ be a generator of the multiplicative group $\F_q^*$, and let $f(x)=x-\alpha$. Then the order of $f(x)$ is clearly $q-1$. Now pick a prime $l$ dividing $q-1$ in the following way: if $q\equiv 3\bmod 4$, choose $l$ to be odd, otherwise choose any $l$. Then by Theorem \ref{order}, the polynomial $f(x^{l^k})$ is irreducible for every $k\in \N$. The set $A\coloneqq \{x,f(x^{l^k})\}_{k\in \N}$ is therefore an infinite, non-trivial $F$-set.
\end{proof}

The reader should notice that the same type of strategy to address the analogous problem over the integers (for additional details see \cite[Section~1]{and}) is beyond the reach of known results. 
In fact, in order to apply the same strategy as in the polynomial case, one would require in particular the existence of a polynomial of the form $f(x)=kx^2+1$, where $k\in \N$, such that $f(n)$ is prime for infinitely many $n$. Unfortunately, the existence of polynomials in $\Z[x]$ of degree $>1$ which assume infinitely many prime values is still an open question (see for example \cite{guy}).

\section{\texorpdfstring{$F$}{}-sets and their width}\label{width}
The examples constructed in the proof of Theorem~\ref{proof_of_conj} are, for $q\neq 2,3$, in some sense ``minimal''. In fact, the set of all the irreducible factors of $f(x)-f(0)$, where $f$ runs over all the elements of the $F$-set, is finite. 
It is therefore natural to ask whether, for every fixed $q$, one can construct an $F$-set in $\F_q[x]$ where the subset of irreducible divisors (of elements of the form $f(x)-f(0)$, for $f$ in the $F$-set) is infinite. This happens for the examples constructed in \cite{and}. The following definitions  formalize the notion of minimality for an $F$-set.
\begin{definition}
Let $A\sub I_q$ be an $F$-set. We define the \emph{nullity} of $A$  as
$$N(A)=\{f(x)\in A\colon f(x)\nmid g(x)-g(0), \,\,\forall \, g(x)\in A\}.$$
\end{definition}
It is easy to check that if $A$ is an $F$-set, then $A\setminus N(A)$ is again an $F$-set. Thus, given an $F$-set $A$, it is possible to define a sequence of $F$-sets as follows:
$$A_0\coloneqq A$$
$$A_n\coloneqq A_{n-1}\setminus N(A_{n-1}) \quad \forall \, n\geq 1.$$
This gives us a filtration on $A$:
$$A_0\supseteq A_1\supseteq \ldots\supseteq A_n\supseteq\ldots$$
which we will call \emph{nullity filtration}.	
\begin{definition}
The minimal $n\in\N$ such that $A_n$ is finite, if it exists, is called \emph{width} of $A$, and is denoted by $w(A)$. If such $n$ does not exist, we set $w(A)=\infty$.
\end{definition}
Notice that an $F$-set $A$ is finite if and only if $w(A)=0$. Therefore, Theorem~\ref{proof_of_conj} can be restated as follows: for every prime power $q$, there exists a non-trivial $F$-set of non-zero width. In particular, the $F$-sets constructed in the proof of the theorem have width $1$ when $q\neq 2,3$, and infinite width when $q=2,3$. It is clear that $F$-sets of width $1$ are in some sense the simplest possible infinite $F$-sets.
\begin{example}
The set $I_q$ of all monic irreducible polynomials in $\F_q[x]$ has infinite width. In fact, let $f(x)\in I_q$ and pick any $a\in \F_q^*$. By Dirichlet's theorem for $\F_q[x]$ (see for example \cite{kor}), there exists at least one (in fact there exist infinitely many) polynomial $g(x)$ such that $h(x)\coloneqq g(x)\cdot xf(x)+a$ is irreducible. Thus, $f(x)\mid h(x)-h(0)$ and this shows that $N(A)=\emptyset$. Therefore we have that $A_n=A$ for every $n$, which implies that $w(A)=\infty$.
\end{example}
The same argument used in the example above shows that if $A$ is an infinite $F$-set, then either $A=I_q$ or $I_q\setminus A$ is infinite. Indeed, suppose that $B\sub I_q$ is a finite set such that $A\sqcup B= I_q$ and let $f(x)\in B$. Fix $a\in\F_q^*$. Since there are infinitely many $g(x)\in\F_q[x]$ such that $g(x)\cdot xf(x)+a\in I_q$ is irreducible, it follows that there are infinitely many $g(x)$ such that $g(x)\cdot xf(x)+a\in A$. But since $A$ is an $F$-set and $f(x)\mid g(x)\cdot xf(x)$, it follows that $f(x)\in A$. Therefore, $B=\emptyset$.

The next proposition recollects some of the basic properties of the nullity and the width of an $F$-set. Notice that any union or intersection of $F$-sets is again an $F$-set.
\begin{proposition}\label{basic_prop}
		Let $A,B$ be $F$-sets, then we have:
		\begin{enumerate}
			\item $N(A)\cup N(B)\sub N(A\cup B)$;
			\item $N(A)\cap N(B)\supseteq N(A\cap B)$;
			\item if $A\sub B$, then $w(A)\leq w(B)$. If moreover $B\setminus A$ is finite, then $w(A)=w(B)$;
			\item if $w(A)$ and $w(B)$ are both finite, then $w(A\cup B)$ is finite;
			\item if $A$ is infinite and $N(A)$ is finite, then $w(A)=\infty$.
		\end{enumerate}
\end{proposition}
\begin{proof}
The claims (1) and (2) follow immediately from the definition of nullity. Let $\{A_n\}_{n\in \mathbb N}$ and $\{B_n\}_{n\in \mathbb N}$ be the nullity filtrations of $A$ and $B$ respectively. To prove (3), first note that $N(B)\cap A\sub N(A)$. Thus, if $f(x)\in A\setminus N(A)$, then $f(x)\notin N(B)$, since otherwise we would have $f(x)\in N(B)\cap A$. This shows that $A\setminus N(A)\sub B\setminus N(B)$. The same argument shows that $A_n\sub B_n$ for every $n\in \N$, and this implies that $w(A)\leq w(B)$. If $|B\setminus A|<\infty$, notice the following:
$$B\setminus N(B)=(A\cup (B\setminus A))\setminus N(B)=(A\setminus N(B))\cup ((B\setminus A)\setminus N(B)).$$
Now $(A\setminus N(B))\setminus (A\setminus N(A))=N(A)\setminus N(B)$, but if $f(x)\in N(A)\setminus N(B)$, then $f(x)\mid g(x)-g(0)$ for some $g\in B\setminus A$, and therefore $N(A)\setminus N(B)$ is finite. This shows that $A\setminus N(B)$, and hence $B\setminus N(B)$, differs from $A\setminus N(A)$ by a finite set. Applying the same argument with $A_n$ and $B_n$ in place of $A$ and $B$ shows that $B_n\setminus A_n$ is finite for all $n\in \N$ and the claim follows.

For point (4), notice first that if $C$ is an $F$-set, then $w(C)$ is infinite if and only if the following holds: for every $t\in \N$ there
exists $r\geq t$ and a set $\{f_1(x),\ldots,f_r(x)\}\sub C$ such that:
\[f_1(x)\neq x\quad \text{and} \quad f_i(x)\mid f_{i+1}(x)-f_{i+1}(0)\quad \forall i\in\{1,\ldots,r\}.\] 
In fact, assume first that $w(C)=\infty$ and let $\{C_n\}_{n\in \N}$ be the nullity filtration of $C$. If there exists $m\in \N$ such that $N(C_m)=\emptyset$, the claim is obvious since then there exists an infinite set $\{f_1(x),\ldots,f_n(x),\ldots\}\sub C$ with $f_i(x)\mid f_{i+1}(x)-f_{i+1}(0)$ for all $i$. Otherwise, fix $t\in \N$ and pick $f_1(x) \in N(C_t)$, so that $f_1(x)\neq x$. Since $f_1(x)\notin N(C_{t-1})$, there exists $f_2(x)\in C_{t-1}$ such that $f_1(x)\mid f_2(x)-f_2(0)$. Now $f_2(x)\notin N(C_{t-2})$, thus there exists $f_3(x)\in C_{t-2}$ such that $f_2(x)\mid f_3(x)-f_3(0)$, and so on until we get a set $\{f_1(x),\ldots,f_t(x)\}$ as required. 
Vice versa, note that if $w(C)<\infty$, then there exists $n\in \N$ such that $C_n=\{x\}$. Therefore no sequence $\{f_1(x),\ldots,f_r(x)\}$ with the property described above can have more than $n$ elements, as the smallest $F$-set containing the sequence is a subset of $C$ and it cannot have larger width. Assume now that $w(A),w(B)<\infty$. If it holds that $w(A\cup B)=\infty$, then for every $t\in \N$ there exists $r\in \N$ such that $r\geq t$ and a set $\{f_1(x),\ldots,f_r(x)\}\sub A\cup B$ as above. Now notice that if $f_r(x)\in A$ (resp. $B$) by definition of $F$-set we have that $f_i(x)\in A$ (resp. $B$) for every $i\leq r$. Since $t$ was arbitrary, this shows that $w(A)=\infty$ or $w(B)=\infty$, contradiction.

Finally, let us prove (5). For $n\geq 1$, let $f(x)\in N(A_n)$. By the definition of nullity, there exists $g(x)\in A_{n-1}$ such that $f(x)\mid g(x)-g(0)$ and $g(x)\in N(A_{n-1})$. This shows that $\deg f(x)$ is strictly smaller than $\deg g(x)$ for all $g(x)\in N(A_{n-1})$. Since $N(A)$ is finite, this argument proves inductively that $N(A_n)$ is finite for every $n$. Consider the sequence defined by 
\[\disp d_n\coloneqq \max_{f\in N(A_n)}\{\deg f(x)\}.\] We have showed that $\{d_n\}_{n\in\N}$ is strictly decreasing; hence there exists $j\in \N$ such that $N(A_j)=\emptyset$. Since $A_j$ differs from $A$ by a finite set, the claim follows by (3).
\end{proof}

	An $F$-set $A$ has width $\leq 1$ if and only if the set $A\setminus N(A)=\{f(x)\in A\colon f(x)\mid g(x)-g(0) \mbox{ for some } g(x)\in A\}$ is finite. It is therefore an interesting task to construct $F$-sets which have width greater than $1$.

\section{Preliminary results}\label{technology}
In this section we prove some ancillary results which will allow the construction of $F$-sets of width strictly greater than $1$.
However, we state them separately, as they might have other applications.
\begin{proposition}\label{irred_pol}
Let $p$ be a prime number. Let $K$ be a field containing a primitive $p$-th root of $1$. Let $f(x)\in K[x]$ be a monic, irreducible polynomial such that $f(0)$ is not a $p$-th power. If $p=2$, assume in addition that $-1$ is a square in $K$ or that $\deg f$ is even. Then for every $k\geq 0$, the polynomial $f(x^{p^k})$ is irreducible.
\end{proposition}
\begin{proof}

We prove the proposition by induction. For $k=0$, there is nothing to do.
Let the claim be true for $0,\ldots,k-1$ and consider $f(x^{p^k})$. 
The proof can be reduced to proving the following statement:
\begin{center}
\emph{if $f(x^{p^k})$ is reducible, then it can be written as $g(x^p)h(x^p)$ with $\deg g,\deg h>0$.} 
\end{center}
Indeed, notice that if the statement above is true, this concludes the proof as $f(x^{p^k})=g(x^p)h(x^p)$ and then setting $x^p=y$ we get $f(y^{p^{k-1}})=g(y)h(y)$, which is a contradiction by the induction hypothesis.

Let now $\xi\in K$ be a primitive $p$-th root of $1$. Suppose one has the factorization $f(x^{p^k})=g(x)h(x)$, with $g,h$ monic, $g$ irreducible and $\deg g,\deg h>0$. Note that $g(x)h(x)=g(\xi x)h(\xi x)$.

	We have to distinguish two cases:

	1) $g(x)$ is of the form $s(x^p)$ for some $s(x)\in K[x]$ of positive degree. Then $f(x^{p^k})=s(x^p)h(x)$, and therefore $h(x)=h(\xi x)$. This shows that $h(x)$ is of the form $t(x^p)$ for some $t(x)$ of positive degree. In this case, we are done.

2) $g(x)$ is not of the form $s(x^p)$. In this case, since $g(x)$ is irreducible, we have that $\gcd(g(\xi^i x),g(\xi^j x))=1$ for every $i,j\in\{0,\ldots, p-1\}$ such that $i\neq j$. In fact, if this was not the case, then we would have $g(x)=g(\xi^i x)$ for some $i\in\{1,\ldots,p-1\}$ and this would imply that $g(x)$ has the form $s(x^p)$ for some $s(x)\in K[x]$ of positive degree, which contradicts the fact that we are in case (2). Now let $i\in\{1,\ldots,p-1\}$. Since $g(x)h(x)=g(\xi^ix)h(\xi^i x)$, it follows that $g(\xi^i x)\mid h(x)$ as $g(\xi^i x)$ is coprime with $g(x)$. As this holds for any $i\in \{1,\dots,p-1\}$, we have $h(x)=g(\xi x)g(\xi^2 x)\ldots g(\xi^{p-1}x)u(x)$ for some $u(x)\in K[x]$, so that
$$f(x^{p^k})=g(x)g(\xi x)\ldots g(\xi^{p-1}x)u(x).$$
Notice that $u(x)=u(\xi x)$, so if $\deg u>0$ we are done again. Assume that this is not the case, i.e.\ let $u(x)=u$ be a constant. If $p>2$, the coefficient of the leading term of $g(x)g(\xi x)\ldots g(\xi^{p-1}x)$ is 
\[\xi^{\deg g\cdot\sum_{i=1}^{p-1}i}=\xi^{\deg g\left(\frac{p-1}{2}\right)p}=1,\] which implies that $u=1$ because $f(x)$ is monic. This yields a contradiction because the constant term of $g(x)g(\xi x)\ldots g(\xi^{p-1}x)$ is a $p$-th power and it coincides with $f(0)$.
If $p=2$, then $u\in\{1,-1\}$ because $f,g,h$ are all monic. If $-1$ is a square in $K$, then the constant term of $g(x)g(-x)u(x)$ is a square in any case, which is a contradiction. If the degree of $f$ is even, then $4\mid \deg f(x^{2^k})$ since $k\geq 1$ and thus $\deg g$ is even, implying that $u=1$ and that again the constant term of $g(x)g(-x)u$ is a square, which is again a contradiction. 
\end{proof}

\begin{remark}
In the case $K=\F_q$ it is easy to see that Proposition~\ref{irred_pol} can be deduced from Theorem~\ref{order}. On the other hand our proposition holds for any field $K$.
\end{remark}

\begin{lemma}\label{cyclotomic_polynomials}
Let $p>2$ be a prime number, $n\in \N_{>0}$ and $q=p^n$. Then we have the following.
\begin{enumerate}
\item Let $a\in \F_q$, let $k$ be a non-negative integer and let $f(x)=x^{2^k}-a\in \F_q[x]$. Then every irreducible factor of $f(x)$ either has degree $1$ or is of the form $x^{2^{t+1}}+bx^{2^t}+c$ for some $t\in \N$ and $b,c\in \F_q$.
\item Let $s,m\in \N$ and $2\nmid m$. Let $g(x)$ be an irreducible polynomial of order $2^s\cdot m$. Then $g(x)$ divides a polynomial of the form $x^{2^k}-a$, for some $k\in \N$ and $a\in \F_q$, if and only if $m\mid q-1$.
\end{enumerate}
\end{lemma}
\begin{proof}
Let us prove (1). If $a=0$ the claim is obvious, therefore suppose $a\in\F_q^*$.

We first show that for any fixed $u\in \F_{q^2}^*$ and non-negative integer $k$, every irreducible factor of 
$g_k(x):=x^{2^k}-u\in \F_{q^2}[x]$ is of the form $x^{2^i}+w$, for some $i\in\N$ and $w\in \F_{q^2}$. Once again, we proceed by induction.
If $k=0$, the claim is trivially true, therefore let us assume it for $k$ and consider $g_{k+1}(x)=x^{2^{k+1}}-u$. 
If $u=w^2$ for some $w\in\F_{q^2}$, then $g_{k+1}(x)=(x^{2^k}+w)(x^{2^k}-w)$, and by the induction hypothesis we are done. On the other hand, if $u$ is not a square in $\F_{q^2}$, then also $-u$ is not a square (as $-1$ is always a square in $\F_{q^2}$) and therefore the polynomial $x^2-u$ is irreducible in $\F_{q^2}[x]$. Thus, the claim follows by Proposition~\ref{irred_pol}.

Now consider $f(x)$ as a polynomial in $\F_{q^2}[x]$. We denote by $\phi_q\colon \F_{q^2}[x]\to\F_{q^2}[x]$ the Frobenius morphism defined by $\sum a_ix^i\mapsto \sum a_i^qx^i$. Let $g(x)$ be an irreducible factor of $f(x)$ in $\F_{q^2}[x]$. Then $f(x)=g(x)h(x)$ for some $h(x)\in \F_{q^2}[x]$ and therefore $f(x)=\phi_q(f(x))=\phi_q(g(x))\phi_q(h(x))$. This shows that $\phi_q(g(x))$ is also a factor of $f(x)$. By what we proved earlier, $g(x)=x^{2^i}+u$ for some $u\in \F_{q^2}$. If $\phi_q(g(x))=g(x)$, this means that $u\in \F_q$, and therefore $g(x)\in \F_q[x]$ is an irreducible factor of $f(x)$ over $\F_q[x]$, and we are done. If $\phi_q(g(x))\neq g(x)$, since both polynomials are monic and $g(x)$ is irreducible over $\F_{q^2}[x]$, it follows that also $\phi_q(g(x))$ is irreducible over $\F_{q^2}[x]$. This shows that $g(x)\phi_q(g(x))$ is an irreducible factor of $f(x)$ over $\F_q[x]$. It is immediate to see that $g(x) \phi_q(g(x))$ has the required form: 
\[g(x)\phi_q(g(x))=x^{2^{i+1}}+(u+u^q)x^{2^i}+uu^q.\]

Now let us prove (2). First recall that, by Theorem~\ref{root}, if $\deg g=t$ and $\alpha$ is a root of $g$, the order of $g$ equals the order of $\alpha$ in the multiplicative group $\F_{q^{2^t}}^*$. Suppose first that $g(x)\mid x^{2^k}-a$, for some $k\in \N$ and $a\in \F_q$. Let $\alpha$ be a root of $g(x)$. Then $\alpha^{2^k}=a$ and therefore there exists $r\in \N$ with $r\mid q-1$ such that $\alpha^{2^kr}=1$, and the claim follows.
Conversely, suppose that $\alpha^{2^sm}=1$. Since $m\mid q-1$ and $\F_{q^{2^t}}^*$ is cyclic, it follows that $\alpha^{2^s}=a$ for some $a \in \F_q$, as there is only one subgroup of order $m$ of $\F_{q^{2^t}}^*$, and it is entirely contained in $\F_q$. It follows that $g(x)\mid x^{2^s}-a$.
\end{proof}

\section{Constructing \texorpdfstring{$F$}{}-sets of width \texorpdfstring{$2$}{} and \texorpdfstring{$\infty$}{}}\label{main_section}
Using the results of the previous section, we now prove a stronger version of Theorem~\ref{proof_of_conj}. In particular, we show that we can always construct an infinite, non-trivial $F$-set $A$ for which the set of prime divisors of all the elements of the form $f(x)-f(0)$ (for $f\in A $) is again infinite.
\begin{theorem}\label{main_theorem}
	Let $p$ be a prime number, $n$ a non negative integer and $q=p^n$. Then:
		\begin{enumerate}
			\item[a)] if $q\neq 2,3$, there exists an $F$-set of width $2$;
			\item[b)] there exists an $F$-set of infinite width in one of the following cases:
				\begin{enumerate}[i)]
					\item $p\equiv 2,5\bmod 9$ and $n=1$;
					\item $p\equiv 5\bmod 8$ and $n$ is odd;
					\item $q\equiv 3\bmod 4$.
			\end{enumerate}
	\end{enumerate}
\end{theorem}
	\begin{proof}
		a) Let us choose a prime $l$ in the following way.
		$$l=\begin{cases}2 & \mbox{if } q\equiv 1\bmod 4\\
						 \mbox{any odd prime dividing } q-1 & \mbox{if } q\equiv 3\bmod 4\\
						3 & \mbox{if } q=4\\
						\mbox{any prime $\geq 5$ dividing } q-1 & \mbox{if } q=2^n, \,\, n\geq 3.
		\end{cases}$$
		Note that a prime as in the fourth case always exists in virtue of Catalan's Conjecture (now Mih\u{a}ilescu's theorem, see \cite{mih}), which states that the only integer solution of the equation $x^a-y^b=1$, with $x,y>0$ and $a,b>1$, is $x=3,y=2,a=2,b=3$.

		We claim that there exist $\alpha,\beta\in \F_q$ such that:
		\begin{itemize}
			\item both $\alpha,\beta$ are not $l$-powers;
			\item the polynomial $x^2+\alpha x+\beta$ is irreducible.
		\end{itemize}
We will show that this is possible for any choice of $l$ as above.
		
		Fix any $\alpha\in \F_q^*$ and consider the bijection
		$$\varphi_{\alpha}\colon\F_q\to\F_q$$
		$$y\mapsto \alpha^2-4y.$$
		
When $l=2$ and $p>2$, notice that $\varphi_{\alpha}(0)$ is a square. On the other hand, if $\gamma$ is not a square, $\varphi_{\alpha}(\gamma)\neq 0$. Since the set of non-zero squares and that of non-squares have the same cardinality, there must be some non-square $\beta$ such that $\varphi_{\alpha}(\beta)$ is not a square.

If $l>2$ and $p>2$, the subset of the elements of $\F_{q}^*$ which are not $l$-powers has cardinality $\disp\frac{l-1}{l}(q-1)$, which is strictly larger than the number of squares in $\F_q^*$. Thus, there exists a non-$l$-power $\beta$ such that $\varphi_{\alpha}(\beta)$ is not a square.
This shows that, chosen any non-$l$-power $\alpha$, there exists a non-$l$-power $\beta$ such that $\alpha^2-4\beta$ is not a square, and therefore the polynomial $x^2+\alpha x+\beta$ is irreducible. 

If $l=3$, $p=2$ and $n=2$, let $\F_{4}=\F_2(\alpha)$, where $\alpha$ is a root of $x^2+x+1$. Then one checks that $\alpha$ is not a cube and $x^2+\alpha x+\alpha$ is irreducible.

Finally, let $p=2$, $n\geq 3$ and $l\geq 5$. The number of monic, irreducible polynomials of degree $2$ in $\F_q[x]$ is $\disp\frac{q^2-q}{2}$. The number of polynomials of the form $x^2+\alpha x+\beta$ where both $\alpha,\beta$ are not $l$-powers is 
\[ \left(q-1-\frac{q-1}{l}\right)^2.\] 
Thus our claim is proved whenever 
\[\frac{q^2-q}{2}+\left(q-1-\frac{q-1}{l}\right)^2>q^2-1,\] 
since $q^2-1$ is the number of all polynomials of the forms $x^2+\alpha x+\beta$, with $(\alpha,\beta)\neq (0,0)$. This inequality is equivalent to
		$$S(q,l)\coloneqq\left(\frac{1}{2}l^2-2l+1\right)q^2+\left(-\frac{5}{2}l^2+4l-2\right)q+(l-1)^2+l^2\geq 0.$$
Let 
\[ A(l)\coloneqq\frac{1}{2}l^2-2l+1, \quad B(l)\coloneqq-\frac{5}{2}l^2+4l-2 \quad \text{and}\quad C(l)\coloneqq(l-1)^2+l^2.\] 
As $l\geq 5$, we have that $A(l)>0$ and 
\[\frac{-B(l)+\sqrt{B(l)^2-4A(l)C(l)}}{2A(l)}<12\] 
which shows that $S(q,l)>0$ whenever $n\geq 4$ and $l\geq 5$. One checks that $S(8,7)=49>0$, and the claim is complete.

The main ingredient of the construction is now ready, as we can always produce an irreducible monic polynomial $f(x)=x^2+\alpha x+\beta$ where $\alpha$ and $\beta$ are not $l$-powers.

Let $A\coloneqq\{x\}\cup\{x^{l^k}+\alpha\}_{k\in \N}\cup\{f(x^{l^k})\}_{k\in \N}$. By Proposition~\ref{irred_pol}, the polynomials $x^{l^k}+\alpha$ and $f(x^{2^k})$ are irreducible for every $k\geq 0$. Thus $A$ is an infinite, nontrivial $F$-set. Note that $N(A)=\{f(x^{l^k})\}_{k\in\N}$ by construction. Thus, $A_1=A\setminus N(A)=\{x\}\cup\{x^{l^k}+\alpha\}_{k\in\N}$ and $A_2=\{x\}$, implying that $w(A)=2$.

		b) When $p\equiv 2,5\bmod 9$, an $F$-set of infinite width is constructed in~\cite[Theorem~1.1]{and}.

		When $p\equiv 5\bmod 8$ and $n$ is odd, $2$ is not a square in $\F_q$ and therefore the $F$-set $A$ constructed in the proof of Theorem~\ref{proof_of_conj} has infinite width, since $N(A)=\emptyset$.

Let now $q\equiv 3 \bmod 4$ and $A\coloneqq \{f\in I_q\colon f\mid x^{2^k}-a\mbox{ for some }k\in \N \mbox{ and }a\in \F_q\}$. By (1) of Lemma~\ref{cyclotomic_polynomials}, this is an infinite, non-trivial $F$-set. Let us prove that $N(A)=\emptyset$, so that $w(A)=\infty$. This amounts to show that for every $f(x)\in A$, there exist $d,e\in \F_q$ and $s\in \N$ such that:
		\begin{itemize}
			\item $f(x)\mid x^{2^{s}}+d$.
			\item $x^{2^{s+1}}+dx^{2^{s}}+e\in A$;			
		\end{itemize}
By construction, $f(x)$ divides a polynomial of the form $x^{2^s}+d$ for some $d\in \F_q$. Hence it is enough to find $e\in \F_q$ such that 
$x^{2^{s+1}}+dx^{2^{s}}+e$ is in $A$.
In order to do so, we first prove a weaker statement and then show that the general fact easily follows by Proposition \ref{irred_pol}.
\begin{center}
\textbf{Claim:} \emph{there exists $e\in \F_q\setminus \F_q^2$ such that $h(x)=x^2+dx+e$ is irreducible and has order $2^l\cdot n$ with $2 \nmid n$ and $n\mid q-1$.}
\end{center}
\begin{proof}[Proof of the claim.]\renewcommand{\qedsymbol}{}
Let $r$ be the largest positive integer such that $2^r\mid q^2-1$. Notice that since $q\equiv 3\bmod 4$, we have that $r\geq 3$. Let $\alpha\in \F_{q^2}^*$ be any element of order $2^r$. Clearly $\alpha$ is not a square as otherwise $2^{r+1}$ would divide $q^2-1$. 
In addition, $\tr(\alpha)$, namely the trace of $\alpha$, is non-zero, since otherwise the minimal polynomial of $\alpha$ would be of the form $x^2+u$, for some $u\in \F_q$. 
This would imply that $\alpha^2\in \F_q$ and this would imply in turn the existence of an element of $\F_q$ of order $2^{r-1}$ with $r-1\geq 2$, which is in contradiction with the assumption $q\equiv 3\bmod 4$. 
On the other hand, since $\alpha$ is not a square in $\F_{q^2}$, its norm $N(\alpha)$ is not a square in $\F_q$ (this is a standard fact for finite fields). 
Let $u\coloneqq\tr(\alpha)$ and consider the element $\displaystyle\beta\coloneqq\frac{d}{u}\alpha$. Then $\tr(\beta)=d$ and $e\coloneqq N(\beta)=\frac{d^2}{u^2}N(\alpha)$ is again not a square in $\F_q$.
Finally, the order of $\beta$ is $2^l\cdot n$ for some $l\in \N$ and $n\mid q-1$ by construction.
This concludes the proof of the claim as $x^2+dx+e$ is the minimal polynomial of $\beta$. 
\end{proof}
Now we are ready to complete the proof. Consider $h(x)=x^{2}+dx+e$ as in the claim: as $e$ is not a square and the degree of $h(x)$ is even, we can apply Proposition~\ref{irred_pol}, getting that $h(x^{2^s})=x^{2^{s+1}}+dx^{2^s}+e$ is irreducible. 
One  observes also that the order of $h(x)$ is $2^{l+s}n$ and $n|q-1$. By Lemma~\ref{cyclotomic_polynomials} it follows that $h(x)$ divides $x^{2^{l+s}}-a$, as required.

%

Notice that if $q=3$, we have two different examples of $F$-sets of infinite width: the one just constructed above and the one described in the proof of Theorem~\ref{proof_of_conj}.
\end{proof}
	It is natural to formulate the following generalization of Conjecture~\ref{congettura}.
	\begin{conjecture}\label{conjwithwidth}
		For every prime power $q$, there exist non-trivial $F$-sets in $\F_q[x]$ of arbitrary width.
\end{conjecture}
\section{Acnowledgements}
The authors want to thank Violetta Weger for useful discussions and suggestions. The second author is thankful to the SNSF grant number 161757.

\bibliographystyle{plain}
\bibliography{biblio}{}

\begin{thebibliography}{1}

\bibitem{and}
Julio~C. Andrade, Steven~J. Miller, Kyle Pratt, and Minh-Tam Trinh.
\newblock Special sets of primes in function fields.
\newblock {\em Integers}, 14:Paper No. A18, 4, 2014.

\bibitem{guy}
Richard~K. Guy.
\newblock {\em Unsolved problems in number theory}.
\newblock Problem Books in Mathematics. Springer-Verlag, New York, third
  edition, 2004.

\bibitem{jon}
Rafe Jones.
\newblock An iterative construction of irreducible polynomials reducible modulo
  every prime.
\newblock {\em J. Algebra}, 369:114--128, 2012.

\bibitem{bos}
Rafe Jones and Nigel Boston.
\newblock Settled polynomials over finite fields.
\newblock {\em Proc. Amer. Math. Soc.}, 140(6):1849--1863, 2012.

\bibitem{kor}
Heinrich Kornblum and Edmund Landau.
\newblock \"{U}ber die {P}rimfunktionen in einer arithmetischen {P}rogression.
\newblock {\em Math. Z.}, 5(1-2):100--111, 1919.

\bibitem{lid}
Rudolf Lidl and Harald Niederreiter.
\newblock {\em Finite fields}, volume~20 of {\em Encyclopedia of Mathematics
  and its Applications}.
\newblock Cambridge University Press, Cambridge, second edition, 1997.
\newblock With a foreword by P. M. Cohn.

\bibitem{mih}
Preda Mih{\u{a}}ilescu.
\newblock Primary cyclotomic units and a proof of {C}atalan's conjecture.
\newblock {\em J. Reine Angew. Math.}, 572:167--195, 2004.

\end{thebibliography}
	
\end{document}